\documentclass[10pt,conference]{ieeeconf}

\usepackage{mathrsfs}
\usepackage{graphics} 
\usepackage{amsmath} 
\usepackage{amssymb}  

\usepackage{amsthm}
\usepackage{cite}
\usepackage{bm}
\usepackage{acronym}
\usepackage{paralist}
\usepackage{float}
\usepackage{color}
\usepackage{epstopdf}
\usepackage{multicol}
\usepackage{tikz}
\usepackage{hyperref}
\usepackage{graphicx}
\usepackage{soul}
\usepackage{mathtools}
\usepackage{siunitx}
\usepackage{empheq}
\usepackage{accents}
\usepackage{subcaption}
\usepackage{mdframed}
\usepackage[export]{adjustbox}
\renewcommand{\theenumi}{\roman{enumi}}

\makeatletter
\hypersetup{colorlinks=true}
\AtBeginDocument{\@ifpackageloaded{hyperref}
  {\def\@linkcolor{blue}
  \def\@anchorcolor{red}
  \def\@citecolor{red}
  \def\@filecolor{red}
  \def\@urlcolor{red}
  \def\@menucolor{red}
  \def\@pagecolor{red}
\begingroup
  \@makeother\`%
  \@makeother\=%
  \edef\x{%
    \edef\noexpand\x{%
      \endgroup
      \noexpand\toks@{%
        \catcode 96=\noexpand\the\catcode`\noexpand\`\relax
        \catcode 61=\noexpand\the\catcode`\noexpand\=\relax
      }%
    }%
    \noexpand\x
  }%
\x
\@makeother\`
\@makeother\=
}{}}
\makeatother

\newtheorem{Theorem}{Theorem}
\newtheorem{Lemma}{Lemma}
\newtheorem{Problem}{Problem}
\newtheorem{Remark}{Remark}

\newtheorem{Assumption}{Assumption}
\newtheorem{Definition}{Definition}

\DeclareMathOperator{\R}{\mathbb R}

\DeclareMathOperator*{\argmin}{arg\,min}

\newcommand{\bequ}{\begin{eqnarray}}
\newcommand{\eequ}{\end{eqnarray}}

\newcommand{\bb}{\boldsymbol}

\IEEEoverridecommandlockouts

\def\BibTeX{{\rm B\kern-.05em{\sc i\kern-.025em b}\kern-.08em
    T\kern-.1667em\lower.7ex\hbox{E}\kern-.125emX}}

\begin{document}

\title{\LARGE \bf Future-Focused Control Barrier Functions for Autonomous Vehicle Control}

\author{Mitchell Black$^1$ \and Mrdjan Jankovic$^2$ \and Abhishek Sharma$^2$ \and Dimitra Panagou$^3$
\thanks{We would like to acknowledge the support of the Ford Motor Company.}
\thanks{$^1$Department of Aerospace Engineering, University of Michigan, 1320 Beal Ave., Ann Arbor, MI 48109, USA; \texttt{mblackjr@umich.edu}.}
\thanks{$^2$Ford Research and Advanced Engineering, 2101 Village Rd., Dearborn, MI 48124, USA; \texttt{$\{$mjankov1, asharm90$\}$@ford.com}.}
\thanks{$^3$Dept. of Robotics and Dept. of Aerospace Engineering,  Univ. of Michigan, Ann Arbor, MI 48109, USA; \texttt{dpanagou@umich.edu}.}
}
\maketitle  


\begin{abstract}\label{sec.abstract}
In this paper, we introduce a class of future-focused control barrier functions (ff-CBF) aimed at improving traditionally myopic CBF based control design and study their efficacy in the context of an unsignaled four-way intersection crossing problem for collections of both communicating and non-communicating autonomous vehicles. Our novel ff-CBF encodes that vehicles take control actions that avoid collisions predicted under a zero-acceleration policy over an arbitrarily long future time interval. In this sense the ff-CBF defines a virtual barrier, a loosening of which we propose in the form of a relaxed future-focused CBF (rff-CBF) that allows a relaxation of the virtual ff-CBF barrier far from the physical barrier between vehicles. We study the performance of ff-CBF and rff-CBF based controllers on communicating vehicles via a series of simulated trials of the intersection scenario, and in particular highlight how the rff-CBF based controller empirically outperforms a benchmark controller from the literature by improving intersection throughput while preserving safety and feasibility. Finally, we demonstrate our proposed ff-CBF control law on an intersection scenario in the laboratory environment with a collection of 5 non-communicating AION ground rovers.
\end{abstract}


\section{Introduction}\label{sec.intro}

Vehicles with autonomous capabilities have grown increasingly prevalent on public roadways in recent years, and growth is forecasted to continue \cite{idc2020forecast}. Intersection scenarios are of keen interest due to the systemic dangers they pose; in fact, according to the U.S. Federal Highway Administration more than 50$\%$ of all fatal and injury crashes occur at intersections \cite{USFHA2021intersection}. Some have proposed alleviating this problem by using a centralized intersection manager (IM) to communicate safe entry/exit times to incoming connected autonomous vehicles (CAVs) \cite{dresner2008multiagent,yang2016isolated,khayatian2019crossroads+}. In contrast to schedulers, controllers offer better real-time reactivity to a dynamic, evolving environment. In the intersection setting, it is critical that control solutions are designed such that the overall system possesses both safety and liveness properties, i.e. that vehicles are able to traverse the intersection safely.

In both centralized and decentralized approaches, a common element in safe controller design is the use of control barrier functions (CBFs) \cite{ames2017control,Xiao2019Decentralized}. CBFs have been shown to be effective in compensating for some potentially unsafe control action in a variety of applications, including autonomous driving \cite{black2020quadratic, Chen2018Obstacle}, robotic manipulators \cite{singletary2019online, rauscher2016constrained}, and quadrotor control \cite{Li2018SafeLearning}. Studies have further demonstrated that CBFs are useful in maintaining safety in the presence of bounded disturbances \cite{jankovic2018robustcbf, deCastro2018rCBF} and model uncertainties \cite{Black2021Fixed, Taylor2019aCBF}. To date, however, a difficulty encountered when using CBF-based approaches is their tendency to myopically focus on present safety, potentially to the detriment of future performance. This drawback can be mitigated in part by using model predictive control (MPC), which solves an optimal control problem over a time horizon and implements the present control solution. While some recent work has demonstrated the efficacy of synthesizing CBFs with MPC frameworks for safe control\cite{Zeng2021MPCcbf}, such controllers often require the solution to a sequence of optimization problems at a given time, where the size of each optimization grows with the look-ahead horizon.

Motivated by these drawbacks, we introduce a new future-focused control barrier function (ff-CBF) for collision avoidance. Its fundamental underlying assumption is that vehicles seek to minimize unnecessary acceleration (or deceleration). This assumption is manifested as a constant velocity prediction of the positions of surrounding vehicles. We then use this to define the predicted minimum inter-agent distance over a future time interval and enforce that this distance remains above a safe threshold. In other words, the ff-CBF defines a zero super-level set containing vehicle states that are guaranteed to remain safe under a zero-acceleration (i.e. constant velocity) control policy over a period of time. It is worth noting that ff-CBFs are related to the predictive CBFs developed in parallel and introduced in \cite{Breeden2022Predictive}, the latter of which take on an increased computational load in exchange for applicability to more general trajectories rather than constant-velocity trajectories. In this sense, ff-CBFs (like predictive CBFs) are related to recent work on the development of backup CBF policies \cite{Gurriet2020Scalable,Chen2021Backup,Singletary2022Onboard}. Unlike backup and predictive CBF policies, however, our ff-CBF does not require numerical integration of the system trajectories forward in time, the computational demands of which also grow with the look-ahead horizon.
This allows us to take predicted future safety into account for the design of present actions while using a computationally-efficient quadratic program-based control law often used for CBF-based safe control \cite{black2020quadratic,Black2021Fixed,ames2017control,rauscher2016constrained}. 

Our future-focused CBF, however, defines a \textit{virtual} barrier which, in practice, may be violated without defying the \textit{physical} barrier between agents. As such, we introduce the notion of a relaxed future-focused control barrier function (rff-CBF) and show that enforcing forward invariance of its zero super-level set allows permeability of the virtual barrier while satisfying the physical one. The rff-CBF, therefore, permits the execution of safe control actions deemed inadmissible by the ff-CBF, resulting in a reduction in conservatism. In a numerical study, we examine the intersection crossing problem over a wide variety of initial conditions and highlight how an rff-CBF based controller provides the safety and performance benefits of ff-CBF based control while improving feasibility properties of a quadratic program (QP) based control law. We then implement the rff-CBF controller on a collection of ground rovers in a lab environment, and demonstrate its success in safely driving non-communicating vehicles through an unsignaled intersection.

The paper is organized as follows. Section \ref{sec.math prelim} introduces some preliminaries, including set invariance and QP-based control. We formalize the problem under consideration in Section \ref{sec.problem} and introduce our future-focused CBF in Section \ref{sec.main results}. Section \ref{sec.case study} contains the results of our simulated and experimental trials, and in Section \ref{sec.conclusion} we conclude with final remarks and directions for future work.


\section{Mathematical Preliminaries}\label{sec.math prelim}

We use the following notation throughout the paper. $\mathbb R$ denotes the set of real numbers. $\|\cdot\|$ represents the Euclidean norm. $\mathcal{C}^r$ is the set of $r$-times continuously differentiable functions in all arguments. We write $\partial S$ for the boundary of a closed set $S$, and $\textrm{int}(S)$ for its interior. A function $\alpha$ is said to belong to class $\mathcal{K}_\infty$ if $\alpha(0)=0$ and $\alpha: \R \rightarrow \R$ is increasing on the interval $(-\infty,\infty)$. The Lie derivative of a function $V:\mathbb R^n\rightarrow \mathbb R$ along a vector field $f:\mathbb R^n\rightarrow\mathbb R^n$ at a point $x\in \mathbb R^n$ is denoted $L_fV(x) \triangleq \frac{\partial V}{\partial x} f(x)$. 

In this paper, we consider a collection of agents, $\mathcal{A}$, each of whose dynamics is governed by the following class of nonlinear, control-affine systems
\begin{equation}\label{eq: nonlinear control-affine system}
    \dot{\bb{x}}_i = f_i(\bb{x}_i(t)) + g_i(\bb{x}_i(t))\bb{u}_i(t), \quad \bb{x}_i(0) = \bb{x}_{i0},
\end{equation}
where $\bb{x}_i \in \R^n$ and $\bb{u}_i \in \mathcal{U}_i \subset \R^m$ denote the state and control vectors respectively for agent $i \in \mathcal{A}$, and where $f_i: \R^n \rightarrow \R^n$ and $g_i: \R^{m\times n} \rightarrow \R^n$ are locally Lipschitz in their arguments and not necessarily homogeneous across agents. The set $\mathcal{U}_i$ denotes the set of admissible control inputs, and it is assumed that $\mathcal{A}$ has cardinality $A$. A subset of agents $\mathcal{A}_c \subseteq \mathcal{A}$ is assumed to be communicating in that they exchange information (e.g. control inputs), whereas the remaining agents $\mathcal{A}_n = \mathcal{A} \setminus \mathcal{A}_c$ are non-communicating.

Given a continuously differentiable function $h_i: \R^n \rightarrow \R$, we define a safe set $S_i$ as
\begin{equation}\label{eq: safe set}
    S_i = \{\bb{x}_i \in \R^n \; | \; h_i(\bb{x}_i) \geq 0\},
\end{equation}
where $\partial S_i = \{\bb{x}_i \in \R^n \; | \; h_i(\bb{x}_i) = 0\}$ and $\textrm{int}(S_i) = \{\bb{x}_i \in \R^n \; | \; h_i(\bb{x}_i) > 0\}$ denote the boundary and interior of $S_i$. The trajectories of \eqref{eq: nonlinear control-affine system} remain safe, i.e. $\bb{x}_i(t) \in S_i$ for all $t \geq 0$, if and only if $S_i$ is \textit{forward-invariant}. The following constitutes a necessary and sufficient condition for forward invariance of a set $S_i$.

\begin{Lemma}[\hspace{-0.3pt}Nagumo's Theorem\cite{BLANCHINI1999SetInvariance}]\label{lma: nagumos thm}
Suppose that there exists $\bb{u}_i \in \mathcal{U}_i$ such that \eqref{eq: nonlinear control-affine system} admits a globally unique solution for each $\bb{x}_i(0) \in S_i$. Then, the set $S_i$ is forward-invariant for the controlled system \eqref{eq: nonlinear control-affine system} if and only if
\begin{equation}\label{eq: forward invariance}
    L_{f_i}h_i(\bb{x_i}) + L_{g_i}h_i(\bb{x_i})\bb{u}_i \geq 0, \; \forall \bb{x}_i \in \partial S_i.
\end{equation}
\end{Lemma}

One way to render a set forward-invariant is to use CBFs in the control design.
\begin{Definition}\hspace{-0.3pt}\cite[Definition 5]{ames2017control}\label{def: cbf}
    Given a set $S_i \subset \R^n$ defined by \eqref{eq: safe set} for a continuously differentiable function $h_i: \R^n \rightarrow \R$, the function $h_i$ is a \textbf{control barrier function} (CBF) defined on a set $D_i$, where $S_i \subseteq D_i \subset \R^n$, if there exists a function $\alpha \in \mathcal{K}_\infty$ such that
    \begin{equation}\label{eq: cbf condition}
        \sup_{\bb{u}_i \in \mathcal{U}_i}\left[L_{f_i}h_i(\bb{x_i}) + L_{g_i}h_i(\bb{x_i})\bb{u}_i\right] \geq -\alpha(h_i(\bb{x_i}))
    \end{equation}
    holds for all $\bb{x}_i \in D_i$.
\end{Definition}

It is evident that \eqref{eq: cbf condition} reduces to \eqref{eq: forward invariance} when $\bb{x}_i \in \partial S_i$, thus if $h_i(\bb{x}(0)) \geq 0$ and $h_i$ is a CBF on $D_i$ then $S_i$ can be rendered forward-invariant. 
As such, it has become popular to include CBF conditions \eqref{eq: cbf condition} as linear constraints in a quadratic program (QP) based control law \cite{ames2017control, black2020quadratic}, etc. When agents in the system are cooperative and communicating, a centralized controller may be deployed as follows
\begin{subequations}\label{eq.centralized_cbf_qp}
\begin{align}
    \bb{u}^* = \argmin_{\bb{u} \in \mathcal{U}} \frac{1}{2}\|\bb{u}&-\bb{u}^0\|^2 \label{qp_J}\\
    \nonumber \textrm{s.t.} \quad \forall i, j &= 1,\hdots, A, \; i \neq j \\
    \phi_i + \gamma_i\bb{u}_i& \geq 0 \label{subeq.centralized_qp_solo_constraint}, \\
    \phi_{ij} + \bb{\gamma}_{ij}\bb{u}& \geq 0 \label{subeq.centralized_qp_joint_constraint}, 
\end{align}
\end{subequations}
where $\bb{u} = [\bb{u}_1,\hdots,\bb{u}_A]^T$ and $\bb{u}^0 = [\bb{u}_1^0,\hdots,\bb{u}_A^0]^T$ denote concatenations of the input and nominal input vectors respectively, and
\begin{subequations}\label{eq.lfh_lgh_abbrev}
    \begin{align}
        \phi_i &= L_{f_i}h_i(\bb{x}_i) + \alpha_i(h_i(\bb{x}_i)) \\
        \gamma_i &= L_{g_i}h_i(\bb{x}_i) \\
        \phi_{ij} &= L_{f_i}h_{ij}(\bb{x}_i, \bb{x}_j) + L_{f_j}h_{ij}(\bb{x}_i, \bb{x}_j) + \alpha_{ij}(h_{ij}(\bb{x}_i, \bb{x}_j)) \\
        \bb{\gamma}_{ij} &= [L_{g_1}h_{ij}(\bb{x}_i,\bb{x}_j) \hdots L_{g_A}h_{ij}(\bb{x}_i,\bb{x}_j)]
    \end{align}
\end{subequations}
where each $\alpha_i, \alpha_{ij} \in \mathcal{K}_\infty$
such that \eqref{subeq.centralized_qp_solo_constraint} represents an agent-specific constraint (e.g. speed limit) and \eqref{subeq.centralized_qp_joint_constraint} represents an inter-agent constraint (e.g. collision avoidance). Note that $\bb{\gamma}_{ij}$ is a row vector of all zeros except indices $i$ and $j$, denoted $\bb{\gamma}_{ij,[i]}$ and $\bb{\gamma}_{ij,[j]}$ respectively. If the agents are non-communicating, however, then a decentralized control law of the following form may be used:
\begin{subequations}\label{eq.decentralized_cbf_qp}
\begin{align}
    \bb{u}_i^* = \argmin_{\bb{u}_i \in \mathcal{U}_i} \frac{1}{2}\|\bb{u}_i&-\bb{u}_i^0\|^2 \label{qp_J}\\
    \nonumber \textrm{s.t.} \; \forall j = 1,\hdots,A, \; i\neq j \\
    \phi_i + \gamma_i\bb{u}_i& \geq 0 \label{subeq.decentralized_qp_solo_constraint}, \\
    \phi_{ij} + \bb{\gamma}_{ij,[i]}\bb{u}_i& \geq 0 \label{subeq.decentralized_qp_joint_constraint}, 
\end{align}
\end{subequations}
where 
\eqref{subeq.decentralized_qp_solo_constraint} and \eqref{subeq.decentralized_qp_joint_constraint} represent agent-specific and inter-agent constraints similar to \eqref{subeq.centralized_qp_solo_constraint} and \eqref{subeq.centralized_qp_joint_constraint}. As noted by \cite{Jankovic2021Collision}, collision avoidance is guaranteed under the centralized control scheme \eqref{eq.centralized_cbf_qp} whenever it is feasible, unlike the decentralized controller \eqref{eq.decentralized_cbf_qp} under which (for a generic CBF $h_{ij}$) no such guarantee exists even when used uniformly by all agents. In Section \ref{sec.case study}, we use forms of \eqref{eq.centralized_cbf_qp} and \eqref{eq.decentralized_cbf_qp} to solve versions of the intersection crossing problem outlined in Section \ref{sec.problem}.

\section{Problem Formulation}\label{sec.problem}

Let $\mathcal{F}$ be an inertial frame with a point $s_0$ denoting its origin. Consider a collection of vehicles $\mathcal{A}$ approaching an unsignaled four-way intersection, where the dynamics of the $i^\textrm{th}$ vehicle are modeled as
\begin{subequations}\label{eq: dynamic bicycle model}
\begin{align}
    \dot{x}_i &= v_{i}\left(\cos{\psi_i} - \sin{\psi_i}\tan{\beta_i}\right) \label{eq: dyn x} \\
    \dot{y}_i &= v_{i}\left(\sin{\psi_i} + \cos{\psi_i}\tan{\beta_i}\right) \label{eq: dyn y} \\
    \dot{\psi}_i &= \frac{v_{i}}{l_r}\tan{\beta_i} \label{eq: dyn psi} \\
    \dot{\beta}_i &= \omega_i \\
    \dot{v}_{i} &= a_{i},
\end{align}
\end{subequations}
where $x_i$ and $y_i$ denote the position of the center of gravity (c.g.) of the vehicle with respect to $s_0$, $\psi_i$ is the orientation of the body-fixed frame, $\mathcal{B}_i$, with respect to $\mathcal{F}$, $\beta_i$ is the slip angle\footnote{The slip angle is the angle between the velocity vector associated with a point in a frame and the orientation of the frame.} of the vehicle c.g. relative to $\mathcal{B}_i$ (we assume $|\beta_i|<\frac{\pi}{2}$), and $v_{i}$ is the velocity of the rear wheel with respect to $\mathcal{F}$. The state of vehicle $i$ is denoted by $\bb{z}_i = [x_i \; y_i \; \psi_i \; \beta_i \; v_{i}]^T$, and the full state is $\bb{z} = [\bb{z}_1 \hdots \bb{z}_A]^T$. The control input is $\bb{u}_i = [\omega_i \; a_{i}]^T$, where $a_{i}$ is the linear acceleration of the rear wheel and $\omega_i$ the angular velocity of the slip angle, $\beta_i$, which is related to the steering angle, $\delta_i$, via $\tan{\beta_i} = \frac{l_r}{l_r+l_f}\tan{\delta_i}$, where $l_f+l_r$ is the wheelbase with $l_f$ (resp. $l_r$) the distance from the c.g. to the center of the front (resp. rear) wheel. The model, depicted in Figure \ref{fig: bicycle model}, is a dynamic extension of the kinematic bicycle model described in \cite[Chapter 2]{Rajamani2012VDC}, and is often used for autonomous vehicles  \cite{Kong2015Kinematic}.

For safety, we consider that each vehicle must 1) obey the road speed limit and drive only in the forward direction, 2) remain inside the road boundaries, and 3) avoid collisions with all vehicles. The satisfaction of requirement 2) can be handled via nominal design of $\omega_i$, whereas we encode 1) and 3) with the following candidate CBFs:
\begin{align}
    h_{s,i}(\bb{z}_i) &= (v_{max} - v_{r,i})(v_{r,i}) \label{eq: CBF speed limit}\\
    h_{0,ij}(\bb{z}_i,\bb{z}_j) &= (x_i - x_j)^2 + (y_i - y_j)^2 - (2R)^2, \label{eq: CBF inter-agent safety}
\end{align}
where $v_{max}$ denotes the speed limit in m/s and $R$ is a safe radius in m. We note that \eqref{eq: CBF inter-agent safety} is widely used in the literature to encode collision avoidance \cite{Santillo2021pcca,Jankovic2021Collision}. Thus, $h_{s,i}$ and $h_{0,ij}$ define the following safe sets at time $t$: $S_{s,i}(t) = \{\bb{z}_i(t): h_{s,i}(\bb{z}_i(t)) \geq 0\}$ and $S_{0,ij}(t) = \{(\bb{z}_i(t),\bb{z}_j(t)): h_{0,ij}(\bb{z}_i(t),\bb{z}_j(t)) \geq 0\}$,
the intersection of which constitutes the safe set for a given vehicle, i.e. 
\begin{equation}
    S_i(t) = \{S_{s,i}(t) \cap S_{0,i}(t)\}, 
\end{equation}
where $S_{0,i}(t) = \bigcap\limits_{j=1,j\neq i}^{N}S_{0,ij}(t) \vspace{2mm}$.

Before introducing the problem under consideration, we note that the dynamics \eqref{eq: dynamic bicycle model} under some predicted control policy $\hat{\bb{u}}_i$ may be expressed as
\begin{equation}\label{eq.zero_accel_dynamics}
    \dot{\hat{\bb{z}}}_i = f_i(\hat{\bb{z}}_i(t)) + g_i(\hat{\bb{z}}_i(t))\hat{\bb{u}}_i, \; \hat{\bb{z}}_i(t_0) = \bb{z}_i(t_0),
\end{equation}
where $\hat{\bb{z}}_i \in \R^n$ denotes the state predicted under the policy $\hat{\bb{u}}_i$. At any time instance, the predicted dynamics \eqref{eq.zero_accel_dynamics} may be propagated forward in time to determine a state prediction at some future time $\tau > t_0$. In this paper, we take $\hat{\bb{u}}_i$ to be the zero-acceleration policy, defined as
$\hat{\bb{u}}_i \triangleq [\hat\omega_i \; \hat a_i]^T = [0 \; 0]^T$.
\begin{Assumption}\label{ass.z0_safe}
    Let $0 < \bar{\tau} < \infty$. For all vehicles $i \in \mathcal{A}$ with dynamics governed by \eqref{eq: dynamic bicycle model}, assume that the predicted closed-loop trajectories of \eqref{eq.zero_accel_dynamics} under the zero-acceleration policy $\hat{\bb{u}}_i$ beginning at $t_0 = 0$ are safe over the interval $\tau \in [0, \bar\tau]$, i.e. $\hat{\bb{z}}_i(\tau) \in S_i(\tau)$ for all $\tau \in [0,\bar{\tau}]$, $\forall i \in \mathcal{A}$.
\end{Assumption}
\begin{figure}[!t]
    \centering
        \includegraphics[width=0.8\columnwidth,clip]{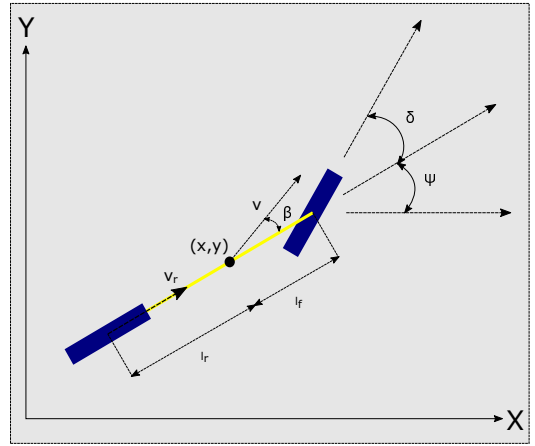}
    \caption{\small{Diagram of bicycle model described in \eqref{eq: dynamic bicycle model}.}}\label{fig: bicycle model}
    \vspace{-3mm}
\end{figure}
Assumption \ref{ass.z0_safe} states that no collisions shall occur between vehicles traveling with constant velocity within a time $\bar\tau$ of the initial time instant, i.e. no vehicles are on a collision course at the outset.
\begin{Problem}\label{prob: main problem}
    Consider a set of vehicles ($i \in \mathcal{A}$) whose dynamics are described by \eqref{eq: dynamic bicycle model}. Given Assumption \ref{ass.z0_safe}, design a control law, $\bb{u}_i^*(t) = [\omega_i^*(t) \; a_{i}^*(t)]^T$, such that, $\forall i \in \mathcal{A}$,
    \begin{enumerate}
        \item the closed-loop trajectories of \eqref{eq: dynamic bicycle model} remain safe for all time ($\bb{z}_i(t) \in S_i(t)$, $\forall t\geq0$), and
        \item at every time $t \geq 0$ the predicted closed-loop trajectories of \eqref{eq.zero_accel_dynamics} over the interval $\tau \in [t, t + \bar\tau]$ remain safe under the zero-acceleration policy $\hat{\bb{u}}_i$, i.e. $\hat{\bb{z}}_i(\tau) \in S_i(\tau)$, $\forall \tau \in [t, t + \bar{\tau}]$, $\forall t \geq 0$ under $\hat{\bb{u}}_i(\tau)$.
    \end{enumerate}
\end{Problem}
The look-ahead time $\bar\tau$ directly influences the set of allowable initial conditions, and vice versa: given $\bar\tau$, the set of allowable initial conditions is restricted to $\mathcal{Z}_0(\bar{\tau}) = \{\bb{z} \in \R^{An}: F(\bb{z},\bar{\tau}) \geq 0\}$, where $F: \R^{An}\times \R_{\geq 0} \rightarrow \R$ is negative if vehicles are predicted to collide under $\hat{\bb{u}_i}$ and non-negative otherwise. On the other hand, given the set of initial states $\mathcal{Z}_0$, the allowable values of $\bar\tau$ are those for which no collisions occur under $\hat{\bb{u}_i}$ over the initial time interval $[0, \bar\tau]$.

In the following section, we introduce a function that serves as a facet of our proposed solution to Problem \ref{prob: main problem}: a future-focused control barrier function (ff-CBF) suitable for QP-based controllers.

\section{Future-Focused Control Barrier Functions}\label{sec.main results}

We first recall the nominal CBF for inter-agent safety \eqref{eq: CBF inter-agent safety}, and note that for two agents $i$ and $j$ it may be rewritten as
\begin{equation}\label{eq: nominal D cbf}
    h_{0,ij}(\bb{z}_i,\bb{z}_j) = D_{ij}^2(\bb{z}_i,\bb{z}_j) - (2R)^2,
\end{equation}
where $D_{ij}(\bb{z}_i,\bb{z}_j) = \sqrt{(x_i - x_j)^2 + (y_i - y_j)^2}$. Let the differential inter-agent position, $\bb{\xi}_{ij}$, velocity, $\bb{\nu}_{ij}$, and acceleration, $\bb{\alpha}_{ij}$, vectors be
\begin{align}
    \bb{\xi}_{ij} &= [\xi_{x,ij}, \; \xi_{y,ij}]^T = [x_i - x_j, \; y_i - y_j]^T,\nonumber\\
    \bb{\nu}_{ij} &= [\nu_{x,ij},\; \nu_{y,ij}]^T = [\dot{x}_i - \dot{x}_j,\; \dot{y}_i - \dot{y}_j]^T, \nonumber\\ 
    \bb{\alpha}_{ij} &= [\alpha_{x,ij}, \;\alpha_{y,ij}]^T = [\ddot{x}_i - \ddot{x}_j, \; \ddot{y}_i - \ddot{y}_j]^T \nonumber,
\end{align}
where we have omitted the argument $t$ for conciseness. In what follows, we also drop the subscript $ij$ from $D$, $\xi$, $\nu$, and $\alpha$. The critical observation is that the inter-agent distance at any arbitrary time, $T$, is $D(\bb{z}_i,\bb{z}_j,T) = \|\bb{\xi}(T)\|$. By assuming zero acceleration, we can use a linear model to predict that at time $T= t + \tau$, we will have that $\bb{\xi}(t+\tau) = \bb{\xi}(t) + \bb{\nu}(t)\tau$, which implies that the predicted distance at a time of $t+\tau$ is 
\small{
\begin{equation}\label{eq: D definition}
    \begin{aligned}
        D(\bb{z}_i,\bb{z}_j,t+\tau) = \sqrt{\xi_x^2 + \xi_y^2 + 2\tau(\xi_x\nu_x + \xi_y\nu_y) + \tau^2(\nu_x^2 + \nu_y^2)}. \nonumber
    \end{aligned}
\end{equation}
}\normalsize
Then, we may define the minimum predicted future distance between agents under the zero-acceleration policy as
\begin{equation}\label{eq: D minimum distance}
    D(\bb{z}_i,\bb{z}_j,t+\tau^*) = \|\bb{\xi}(t) + \bb{\nu}(t)\tau^*\| ,
\end{equation}
where
\begin{equation}\label{eq: tau star}
    \tau^* = \argmin_{\tau \in \R}D^2(\bb{z}_i,\bb{z}_j,t+\tau) = -\frac{\xi_x\nu_x + \xi_y\nu_y}{\nu_x^2 + \nu_y^2}.
\end{equation}

We now introduce our future-focused CBF for collision avoidance, the effect of which is depicted in Figure \ref{fig: ffcbf}:
\begin{equation}\label{eq: ffcbf}
    h_{\hat\tau,ij}(\bb{z}_i,\bb{z}_j) = D_{ij}^2(\bb{z}_i,\bb{z}_j,t+\hat{\tau}) - (2R)^2,
\end{equation}
where 
\begin{equation}\label{eq: tau hat}
    \hat{\tau} = \hat{\tau}^*K_0(\hat{\tau}^*) + (\bar{\tau} - \hat{\tau}^*) K_{\bar{\tau}}(\hat{\tau}^*),
\end{equation}
with $\bar{\tau}>0$ representing the length of the look-ahead horizon, $K_\delta(s) = \frac{1}{2} + \frac{1}{2}\tanh{\left(k(s - \delta)\right)}$, $k>0$, and 
\begin{equation}\label{eq: tau hat star}
    \hat{\tau}^* = -\frac{\xi_x\nu_x + \xi_y\nu_y}{\nu_x^2 + \nu_y^2 + \varepsilon},
\end{equation}
where $0 < \varepsilon \ll 1$.
Using \eqref{eq: tau hat} alleviates undesirable characteristics of \eqref{eq: tau star}, namely that $\tau^*$ may become unbounded. The inclusion of $\varepsilon$ makes \eqref{eq: tau hat star} well-defined, and $K_\delta(t)$ allows \eqref{eq: tau hat} to smoothly approximate $\hat{\tau}^*$ between $0$ and $\bar\tau$. 

It is worth mentioning that the ff-CBF is related to the backup CBFs used for safe control design in \cite{Gurriet2020Scalable,Chen2021Backup} in the following sense: whereas our formulation seeks to ensure that vehicles preserve safety in the future under the zero-acceleration policy $\hat{\bb{u}}_i$, past works have required a backup policy to actively intervene, e.g. to apply proportional braking, in order to guarantee safety.

\begin{figure}[!t]
    \centering
        \includegraphics[width=0.95\columnwidth,clip]{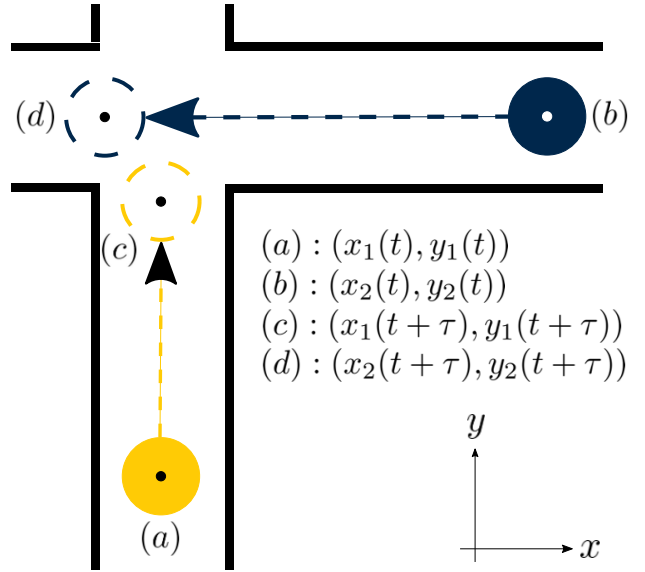}
    \caption{\small{Visualization of the effect of the ff-CBF. Whereas $h_{0,12}$ is evaluated based on the locations of vehicles $1$ and $2$ at time $t$, i.e. $(a)$ and $(b)$, $h_{\tau,12}$ judges safety based on the predicted future locations of the vehicles at time $t + \tau$, i.e. $(c)$ and $(d)$, allowing the present control to take action to avoid predicted future danger.}}\label{fig: ffcbf}
    \vspace{-3mm}
\end{figure}
\begin{Theorem}\label{thm: differentiable ffcbf}
    Consider two agents governed by the dynamics \eqref{eq: dynamic bicycle model} whose states are $\bb{z}_i$ and $\bb{z}_j$. Suppose that $h_{\hat\tau,ij}$ is defined by \eqref{eq: ffcbf}, with $\hat\tau$ given by \eqref{eq: tau hat}. Then, the following hold for all bounded $\bb{z}_i$, $\bb{z}_j$:
    \begin{enumerate}
        \item $h_{\hat\tau,ij} \in \mathcal{C}^1$ 
        \item $h_{\hat\tau,ij} \leq h_{0,ij}$ whenever $\hat\tau \leq 2\hat\tau^*$
    \end{enumerate}
\end{Theorem}
\begin{proof}
For the first part, we must show that the derivative of \eqref{eq: ffcbf} is well-defined and continuous. Consider that from \eqref{eq: D minimum distance}, \eqref{eq: ffcbf}, and \eqref{eq: tau hat} we have
\begin{equation}\label{eq: dot h hat tau}
    \begin{aligned}
        \dot{h}_{\hat\tau,ij}(\bb{z}) &= 2\xi_x\nu_x + 2\xi_y\nu_y + 2\dot{\hat{\tau}}(\xi_x\nu_x + \xi_y\nu_y) \\
        &\quad + 2\hat{\tau}(\nu_x^2 + \nu_y^2 + \xi_x\alpha_x + \xi_y\alpha_y) \\
        &\quad + 2\hat{\tau}\dot{\hat{\tau}}(\nu_x^2 + \nu_y^2) + 2\hat{\tau}^2(\nu_x\alpha_x + \nu_y\alpha_y).
    \end{aligned}
\end{equation}
Since $\hat\tau$ is bounded by definition, it follows that $h_{\hat\tau,ij} \in \mathcal{C}^1$ when $\hat\tau \in \mathcal{C}^1$. From \eqref{eq: tau hat}, we have
\begin{equation}
\small{
\begin{aligned}
    \dot{\hat{\tau}} &=  \dot{\hat{\tau}}^*\left(K_0(\hat{\tau}^*) - K_{\bar{\tau}}(\hat{\tau}^*) \right) + \hat{\tau}^*(\dot{K}_0(\hat{\tau}^*) - \dot{K}_{\bar{\tau}}(\hat{\tau}^*)) + \bar\tau \dot{K}_{\bar\tau}(\hat\tau^*),\nonumber
\end{aligned}}\normalsize
\end{equation}
where
\begin{equation}
    \dot{K}_{\delta}(\hat\tau^*) = \dot{\hat{\tau}}^*\frac{k}{2}\textrm{sech}^2\left(k(\hat\tau^* - \delta)\right) \nonumber
\end{equation}
and from \eqref{eq: tau hat star}
\begin{equation}
    \dot{\hat{\tau}}^* = -\frac{\alpha_x(2\nu_x\tau^* + \xi_x) + \alpha_y(2\nu_y\tau^* + \xi_y) + \nu_x^2 + \nu_y^2}{\nu_x^2 + \nu_y^2 + \varepsilon}  \nonumber
\end{equation}
since $\dot{\hat{\tau}}^*$ and $\dot{K}_{\delta}(t)$ are bounded and continuous for bounded arguments, we have that $\hat{\tau} \in \mathcal{C}^1$ for bounded $\bb{z}_i$, $\bb{z}_j$. Thus, $h_{\hat\tau,ij} \in \mathcal{C}^1$.

For the second part, we observe that $h_{\hat\tau,ij}(\bb{z}) = h_{0,ij}(\bb{z}) + 2\hat\tau(\xi_x\nu_x + \xi_y\nu_y) + \hat\tau^2(\nu_x^2 + \nu_y^2)$, thus $h_{\hat\tau,ij}(\bb{z}_i,\bb{z}_j) \leq h_{0,ij}(\bb{z}_i,\bb{z}_j)$ whenever
\begin{equation}\label{eq: h term negative}
    \hat\tau \leq -2\frac{\xi_x\nu_x + \xi_y\nu_y}{\nu_x^2 + \nu_y^2} = 2\tau^*.
\end{equation}
With $\varepsilon$ in the denominator of \eqref{eq: tau hat star}, it follows that $\hat\tau^* < \tau^*$ whenever $\tau^* > 0$ (and $\hat\tau^* = 0$ when $\tau^* = 0$), thus the inequality in \eqref{eq: h term negative} holds whenever $\hat\tau \leq 2\hat\tau^*$. It follows, then, that $h_{\hat\tau,ij}(\bb{z}) \leq h_{0,ij}(\bb{z})$ whenever $\hat\tau \leq 2\hat\tau^*$.
\end{proof}

\begin{Remark}
The condition $\hat\tau \leq 2\hat\tau^*$ may be satisfied $\forall \hat\tau^* \geq 0$ for choices of $k \geq 1$ in $K_{\delta}(t)$. Higher values of $k$ lead to smaller approximation error $e_{\tau} = |\hat\tau - \tau^*|$ for $\tau^* \in [0,T]$, and thus in practice we use $k = 1000$.
\end{Remark}

Since $h_{\hat\tau,ij} \in \mathcal{C}^{1}$, we have by Definition \ref{def: cbf} that if there exists a function $\alpha \in \mathcal{K}_\infty$ such that \eqref{eq: cbf condition} holds then $h_{\hat\tau,ij}$ is a valid CBF. Under such conditions, our ff-CBF may be synthesized with any nominal control law using \eqref{eq.centralized_cbf_qp} for communicating agents or \eqref{eq.decentralized_cbf_qp} for non-communicating agents. Notably, in contrast to when used with a generic CBF the decentralized control law \eqref{eq.decentralized_cbf_qp} guarantees collision avoidance under our ff-CBF $h_{\hat\tau,ij}$ and dynamics \eqref{eq: dynamic bicycle model} (as long as it is feasible) provided that all vehicles deploy \eqref{eq.decentralized_cbf_qp} with $h_{\hat\tau,ij}$ and are not turning, i.e. $\psi_i = \beta_i = 0$. This is due to the fact that $L_fh_{\hat\tau,ij} \rightarrow 0$ occurs\footnote{In our simulations, we have found that $L_fh_{\hat\tau,ij}$ is on the order of the approximation error $e_{\tau} = |\hat\tau - \tau^*| \approx 10^{-9}$ for $\tau^* \in [0,\bar\tau]$, which may be accounted for by subtracting $\varepsilon \approx 10^{-9}$ from the left-hand side of \eqref{eq.decentralized_safety}.} as $\hat\tau \rightarrow \tau^*$, in which case \eqref{subeq.decentralized_qp_joint_constraint} becomes
\begin{equation}\label{eq.decentralized_safety}
    L_{g_i}h_{\hat\tau,ij}\bb{u}_i + \alpha_{ij}(h_{\hat\tau,ij}) \geq 0, \quad \forall i \in \mathcal{A},
\end{equation}
which, for any given two agent pair $i,j$ yields
\begin{equation}\nonumber
    \dot{h}_{\hat\tau,ij} = L_{g_i}h_{\hat\tau,ij}\bb{u}_i + L_{g_j}h_{\hat\tau,ji}\bb{u}_j \geq -\alpha_{ij}(h_{\hat\tau,ij}) - \alpha_{ji}(h_{\hat\tau,ji})
\end{equation}
where $h_{\hat\tau,ij}=h_{\hat\tau,ji}$, which satisfies the forward invariance condition \eqref{eq: forward invariance} and thus prevents collisions. Intuitively, a zero CBF drift term i.e. $L_fh_{\hat\tau,ij}=0$ is explained by the fact that the ff-CBF $h_{\hat\tau,ij}$ is already predicting the future minimum distance between vehicles $i$ and $j$ under zero-acceleration policies, thus in the absence of an acceleration input the prediction is correct and the minimum distance between vehicles is reached at time $t + \hat\tau$.

It is important to note that the zero level set defined by candidate CBF $h_{\hat\tau,ij}$ represents a \textit{virtual} barrier. Specifically, $h_{\hat\tau,ij}(\bb{z}_i,\bb{z}_j)< 0$ does not imply that a collision has occurred ($h_{0,ij}(\bb{z}_i,\bb{z}_j) < 0$), nor does it suggest that one is unavoidable; rather, $h_{\hat\tau,ij}(\bb{z}_i,\bb{z}_j)<0$ implies that a future collision will occur if the zero-acceleration control policy, $\hat{\bb{u}}_k$, is applied uniformly by each vehicle $k \in \{i,j\}$. In this sense, it is conservative. This motivates the notion of the relaxed future-focused control barrier function (rff-CBF):
\begin{equation}\label{eq.robust_virtual_cbf}
    H_{ij}(\bb{z}_i,\bb{z}_j) = h_{\hat\tau,ij}(\bb{z}_i,\bb{z}_j) + \alpha_0\left(h_{0,ij}(\bb{z}_i,\bb{z}_j)\right),
\end{equation}
where $\alpha_0 \in \mathcal{K}_\infty$. The zero super-level set of $H_{ij}$ is then
\begin{equation}
    S_{H,ij} = \{(\bb{z}_i,\bb{z}_j) \in \R^{2n} \; | \; H_{ij}(\bb{z}_i,\bb{z}_j) \geq 0\},
\end{equation}
which defines a \textit{relaxed} virtual barrier, i.e. one that is less restrictive than the ff-CBF barrier while maintaining the guarantee of collision avoidance. This is proved in the following result.
\begin{Theorem}\label{thm.rv_safety}
    Consider two agents, each of whose dynamics are described by \eqref{eq: nonlinear control-affine system}. Suppose that $H_{ij}$ is given by \eqref{eq.robust_virtual_cbf}, and that $H_{ij} \geq 0$ at $t=0$. If there exist control inputs, $\bb{u}_i$ and $\bb{u}_j$, such that the following condition holds, for all $t\geq 0$,
    \begin{equation}\label{eq.robust_virtual_cbf_condition}
    \begin{aligned}
        \sup_{\substack{\bb{u}_i \in \mathcal{U}_i\\ \bb{u}_j\in\mathcal{U}_j}}\left[L_{f_i}H_{ij}+ L_{f_j}H_{ij}+L_{g_i}H_{ij}\bb{u}_i + L_{g_j}H_{ij}\bb{u}_j\right] \geq 0,
    \end{aligned}
    \end{equation}
    for all $\bb{z} \in \partial S_{H,ij}$, then, the \textit{physical} safe set defined by $S_{0,ij}(t) = \{(\bb{z}_i,\bb{z}_j) \in \R^{2n}\; | \; h_{0,ij}(\bb{z}_i,\bb{z}_j) \geq 0\}$ is forward-invariant under $\bb{u}_i$, $\bb{u}_j$, i.e. there is no collision between agents $i$ and $j$.
\end{Theorem}
\begin{proof}
    In order to show that $S_{0,ij}$ is rendered forward-invariant by \eqref{eq.robust_virtual_cbf_condition}, we must show that \eqref{eq.robust_virtual_cbf_condition} implies that $\dot{h}_{0,ij} \geq 0$ whenever $h_{0,ij} = 0$. We will prove this by contradiction. 
    
    Suppose that $H_{ij},h_{0,ij}=0$, and that \eqref{eq.robust_virtual_cbf_condition} holds but $\dot{h}_{0,ij} < 0$. Then, it follows that $\dot{h}_{0,ij} = 2(\xi_x\nu_x + \xi_y\nu_y) < 0$, which by \eqref{eq: tau hat} implies that
    $\hat\tau>0$.
    With $\hat\tau>0$, it follows that $h_{\hat\tau,ij} < h_{0,ij} = 0$. However, we have assumed that $H_{ij},h_{0,ij} = 0$, which means by definition that $h_{\hat\tau,ij}=0$. Thus, we have reached a contradiction. It follows, then, that \eqref{eq.robust_virtual_cbf_condition} implies that $\dot{h}_{0,ij} \geq 0$ whenever $h_{0,ij} = 0$. As such, $S_{0,ij}$ is rendered forward-invariant.
\end{proof}
As a result of Theorem \ref{thm.rv_safety}, we can use \eqref{eq.robust_virtual_cbf} to encode safety in the context of a CBF-QP control scheme \eqref{eq.centralized_cbf_qp} or \eqref{eq.decentralized_cbf_qp}. In the ensuing section, we conduct a comparative study on the efficacy of the nominal \eqref{eq: nominal D cbf}, future-focused \eqref{eq: ffcbf}, and relaxed future-focused \eqref{eq.robust_virtual_cbf} CBFs across randomized trials of an automotive intersection crossing problem.


\section{Intersection Case Studies}\label{sec.case study}
In this section, we illustrate the use of our future-focused CBFs for collections of both communicating and non-communicating vehicles in the context of simulated and experimental trials of an unsignaled intersection scenario. We provide code and a selection of videos for both on Github\footnote{Link to Github repository: \href{https://github.com/6lackmitchell/ffcbf-control}{github.com/6lackmitchell/ffcbf-control}}.

\subsection{Centralized Control: Simulated Trials}
In an empirical study on a simulated 4-vehicle unsignaled intersection scenario, we illustrate how using a rff-CBF to control communicating vehicles in a centralized manner improves intersection throughput with promising empirical results on safety and QP feasibility.
We study the varying degrees of success of three different centralized controllers of the form \eqref{eq.centralized_cbf_qp} to solve Problem \ref{prob: main problem}, namely to find
\begin{equation}\label{eq.proposed_controller}
    \bb{u}_i^* = [\omega_i^* \; a_{i}^*]^T, \; \forall i = 1,\hdots,A ,
\end{equation}
where the turning rate is
\begin{equation}
     \omega_i^* = \min\left(\max(\omega_i^0,-\bar{\omega}),\bar{\omega}\right),
\end{equation}
and the accelerations $a_1^*,\hdots,a_A^*$ are computed via 
\begin{subequations}\label{eq.proposed_centralized_cbf_qp}
\begin{align}
    [a_{1}^*\hdots a_{A}^*]^T = \argmin_{[a_1\hdots a_A]} \frac{1}{2}\sum_{i=1}^{A}(a_i&-a_i^0)^2\\
    \textrm{s.t.} \quad \forall i, j=1,\hdots,A, \; j\neq i\nonumber \\
    Aa_i &\leq b, \label{eq.qp_input_constraints}\\ 
    \phi_i + \gamma_ia_i &\geq 0, \label{eq.qp_speed_limit}\\
    \phi_{ij} + \bb{\gamma}_{ij,[i]}a_i +\bb{\gamma}_{ij,[j]}a_j &\geq 0, \label{eq.qp_inter-agent_safety}
\end{align}
\end{subequations}
where $\omega_i^0$ and $a_{i}^0$ denote the nominal inputs computed using LQR (see Appendix \ref{app.LQR} for a detailed explanation), \eqref{eq.qp_input_constraints} encodes input constraints of the form $-\bar{a} \leq a_i \leq \bar{a}$, \eqref{eq.qp_speed_limit} enforces both the road speed limit and requires that vehicles do not reverse, and \eqref{eq.qp_inter-agent_safety} is the collision avoidance condition, where $\phi$ and $\gamma$ are as in \eqref{eq.lfh_lgh_abbrev}. Specifically, the controllers under examination are \eqref{eq.proposed_controller} with 
\begin{enumerate}
    \item 0-CBF: $h_{ij} = h_{0,ij}$ according to \eqref{eq: nominal D cbf}
    \item ff-CBF: $h_{ij} = h_{\hat\tau,ij}$ from \eqref{eq: ffcbf}
    \item rff-CBF: $h_{ij} = H_{ij}$ via \eqref{eq.robust_virtual_cbf}
\end{enumerate}
with $\alpha_0(h_{0,ij})=k_0h_{0,ij}$, where $k_0=0.1\max(\hat\tau-1,\varepsilon)$, $\varepsilon=0.001$, the look-ahead horizon $\bar{\tau}=5$s, and $\alpha_{ij}(h_{ij}) = 10h_{ij}$, $\bar{\omega}=\pi/2$, and $\bar{a}_r=9.81$ for all cases. We note that \eqref{eq.proposed_controller} is centralized in the sense that it is assumed that all states, $\bb{z}_i$, and nominal control inputs, $\bb{u}_i^0$, are known. 

For each study, we performed $N=1000$ trials of simulated trajectories of 4 vehicles approaching the intersection from different lanes, all of whose dynamics are described by \eqref{eq: dynamic bicycle model}, using the control scheme described by \eqref{eq.proposed_controller} and a timestep of $dt=0.01$s. At the beginning of each trial, the vehicles were assigned to a lane and their initial conditions were randomized via
\begin{align}
    d_i &= d_0 + U(-\Delta_d,\Delta_d), \nonumber\\
    s_i & = s_0 + U(-\Delta_s,\Delta_s), \nonumber
\end{align}
where $d_i$ denotes the initial distance of vehicle $i$ from the intersection and $s_i$ its initial speed. We chose $d_0=12$m, $\Delta_d=5$m, $s_0=6$m/s, and $\Delta_s=3$m/s, and let $U(a,b)$ denote a sample from the uniform random distribution between $a$ and $b$. For the speed limit, we chose $v_{max}=10$m/s. 

For performance evaluation, we introduce some metrics:
\begin{enumerate}
    \item Success: $\frac{\textrm{Number of Successful Trials}}{\textrm{Number of Trials}}$,\vspace{1.5mm}
    \item Feas.: $\frac{\textrm{Number of Trials in which QP is Always Feasible}}{\textrm{Number of Trials}}$,\vspace{1.5mm}
    \item DLock: $\frac{\textrm{Number of Trials in which Vehicles become deadlocked}}{\textrm{Number of Trials}}$,\vspace{1.5mm}
    \item Unsafe: $\frac{\textrm{Number of Trials Vehicles in which } h_{0,ij}<0}{\textrm{Number of Trials}}$,
\end{enumerate}
where a successful trial is characterized as one where all vehicles exit the intersection at their desired location, a deadlock is characterized as when all vehicles have stopped and remained stopped for 3 sec, and we define ``Avg. Time" as the average time in which the final vehicle reached the intersection exit over all successful trials. 

We examined the performance of each controller under two circumstances: 1) each vehicle seeks to proceed straight through the intersection without turning, and 2) three vehicles seek to proceed straight without turning and one seeks to make a left turn. The results for the 3 different controllers are compiled in Tables \ref{tab.cbf_performance} and \ref{tab.cbf_performance_turning} respectively. Although the 0-CBF in a centralized QP-based control law is known to guarantee safety and QP feasibility under certain conditions \cite{Jankovic2021Collision}, such a controller has no predictive power and is therefore prone to deadlocks.
\begin{table}[!t]
    \vspace{2mm}
    \centering
    \caption{Controller Performance -- All Proceed Straight}\label{tab.cbf_performance}
    \begin{tabular}{|c|c|c|c|c|c|c|c|}
        \hline
        CBF & Success & Feas. & DLock & Unsafe & Avg. Time \\ \hline
        $h_{ij}=h_{0,ij}$ & 0.653 & 1 & 0.347 & 0 & 5.67 \\ \hline
        $h_{ij}=h_{\hat\tau,ij}$ & 1 & 1 & 0 & 0 & 3.45 \\ \hline
        $h_{ij}=H_{ij}$ & 1 & 1 & 0 & 0 & 3.21 \\ \hline
    \end{tabular}\par
\end{table}
\begin{table}[!t]
    \vspace{2mm}
    \centering
    \caption{Controller Performance -- One Left Turn}\label{tab.cbf_performance_turning}
    \begin{tabular}{|c|c|c|c|c|c|c|c|}
        \hline
        CBF & Success & Feas. & DLock & Unsafe & Avg. Time \\ \hline
        $h_{ij}=h_{0,ij}$ & 0.689 & 1 & 0.311 & 0 & 7.75 \\ \hline
        $h_{ij}=h_{\hat\tau,ij}$ & 0.963 & 0.963 & 0 & 0 & 5.33 \\ \hline
        $h_{ij}=H_{ij}$ & 1 & 1 & 0 & 0 & 4.91 \\ \hline
    \end{tabular}\par
\end{table}
We illustrate such a deadlock in Figure \ref{fig.nominal_and_ff_cbf_results}{\color{blue}a}. The ff-CBF-based controller succeeded as long as it was feasible, offering a 39\% reduction in average time over the 0-CBF in the straight scenario and an 31\% time improvement in the turning scenario, but suffered from virtual constraint violations leading to QP infeasibility in the case of turning vehicles, one example of which is shown in Figure \ref{fig.nominal_and_ff_cbf_results}{\color{blue}b}. The rff-CBF controller enjoyed both the same empirical feasibility and safety as the 0-CBF design and improved the average success time to a similar extent as the ff-CBF, specifically by 43$\%$ and 36$\%$ for the straight and turning scenarios respectively. In addition, the rff-CBF control scheme achieved 100$\%$ feasibility even in the turning scenario, despite the constant velocity prediction model not taking a change of heading into account. We leave any theoretical guarantees of feasibility, however, to future work. The state, control, and rff-CBF trajectories for a turning trial are illustrated in Figures \ref{fig.rvcbf_state_trajectories}-\ref{fig.rvcbf_trajectories}. It can be seen from Figure \ref{fig.rvcbf_control_trajectories} that the control actions smoothly take action in advance of any dangerous scenario, and from Figure \ref{fig.rvcbf_trajectories} that both $H_{ij}$ and $h_{0,ij}$ remain non-negative for all $i$, $j$. 

\begin{figure}[!htp]
\vspace{2mm}

\subfloat[]{%
  \includegraphics[clip,trim={0 0 0 15mm},width=0.98\columnwidth]{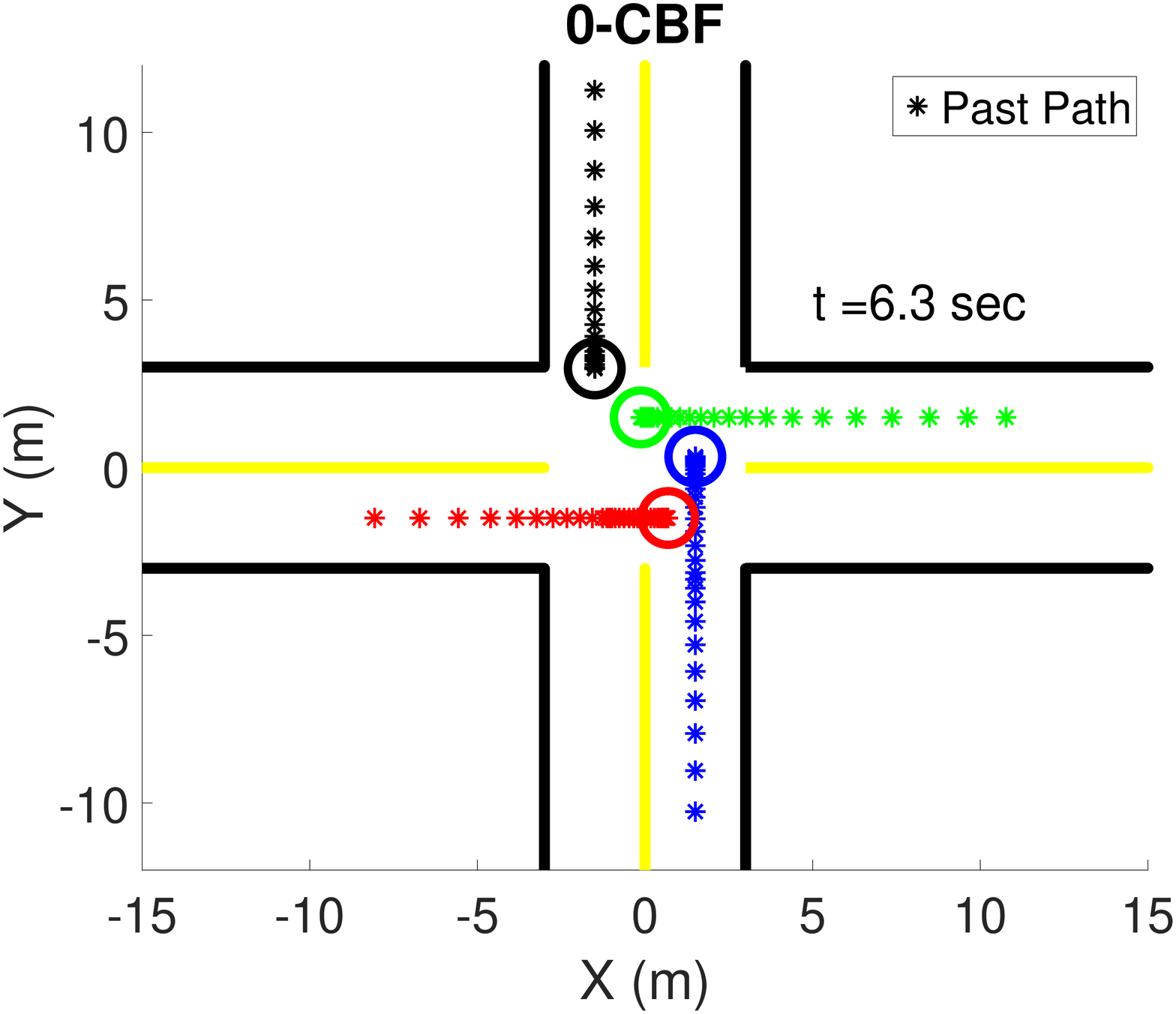}%
}
\vspace{1mm}

\subfloat[]{%
  \includegraphics[clip,trim={0 0 0 15mm},width=0.98\columnwidth]{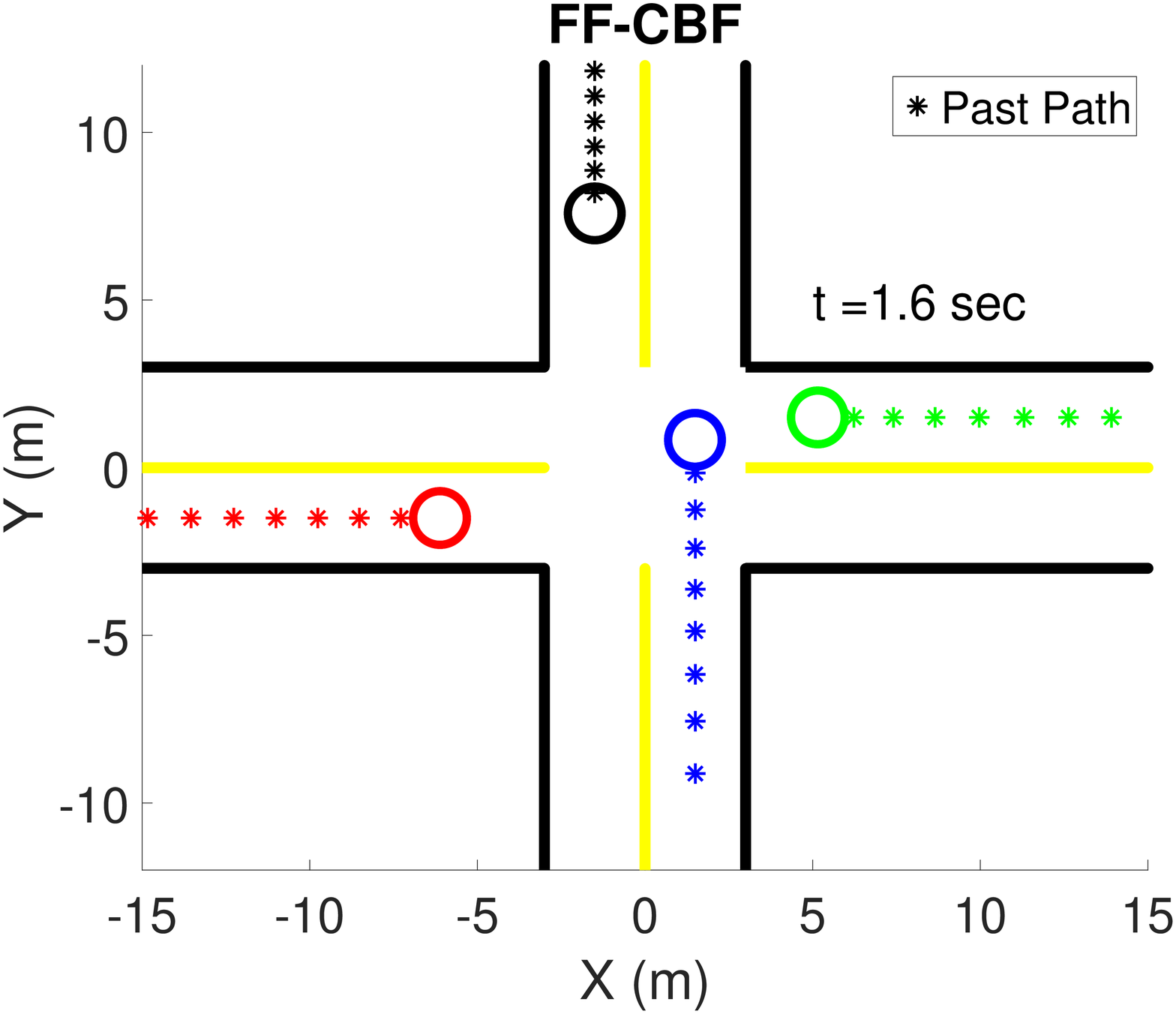}%
}
\vspace{1mm}

\caption{\small{Selected XY trajectories for the intersection crossing problem using (a) 0-CBF (Straight Trial 582) and (b) ff-CBF (Turning Trial 137). In (a), the centralized controller has no predictive power and the vehicles deadlock, whereas in (b) the virtual barrier between blue and black vehicles is violated despite a wide physical margin as the blue vehicle begins to turn left.}}\label{fig.nominal_and_ff_cbf_results}
\vspace{-5mm}
\end{figure}

\begin{figure}[!h]
    \vspace{2mm}
    \begin{minipage}[b]{1\columnwidth}
        \centering
        \includegraphics[width=1\columnwidth,clip,trim={0 0 0 15mm}]{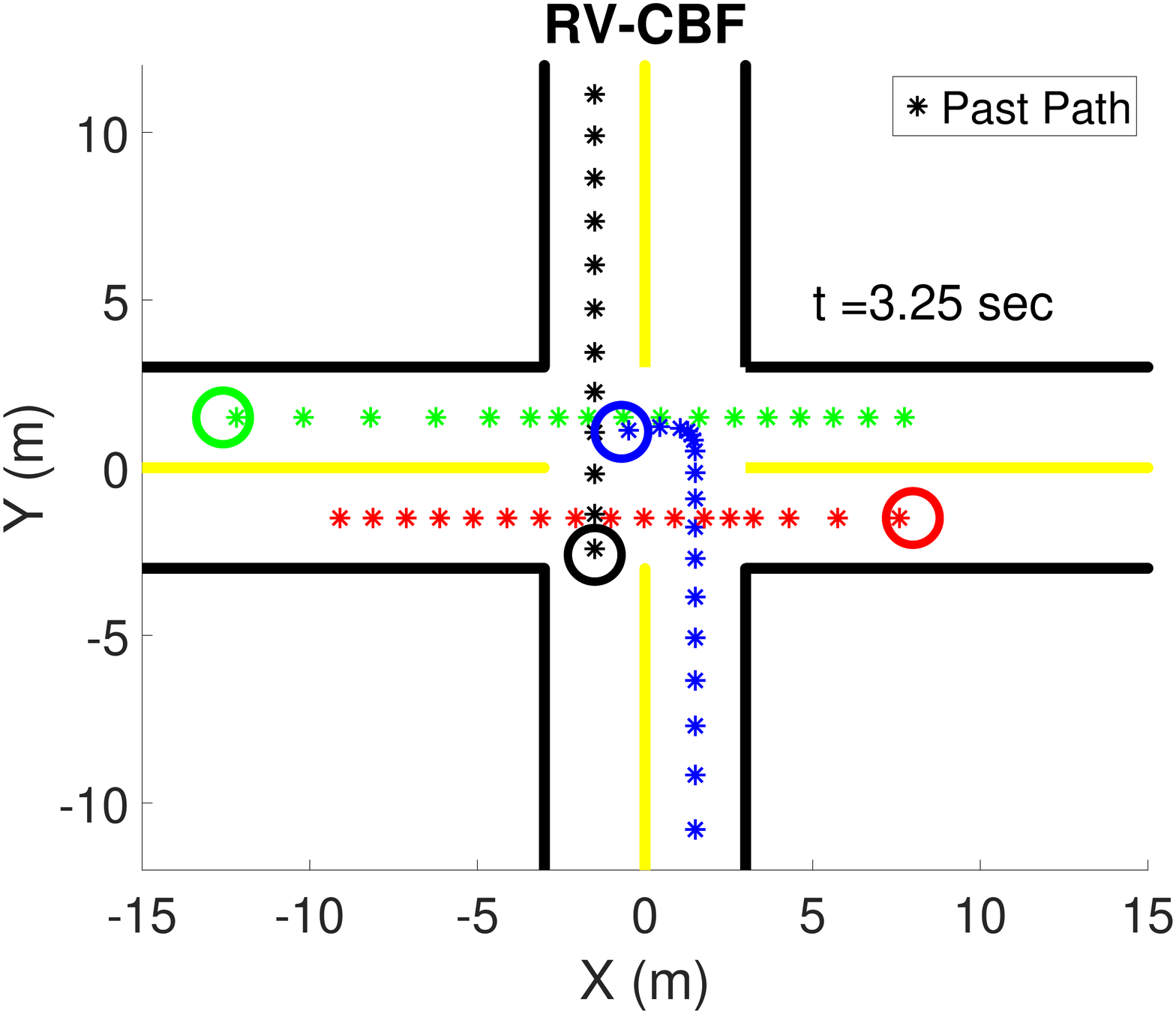}
        \caption{\small{XY trajectories for Trial 650 of the rff-CBF simulation set.}}
        \label{fig.rvcbf_state_trajectories}
        \vspace{5mm}
    \end{minipage}
    \begin{minipage}[b]{1\columnwidth}
        \centering
        \includegraphics[width=1\columnwidth,clip,trim={0 0 0 12mm}]{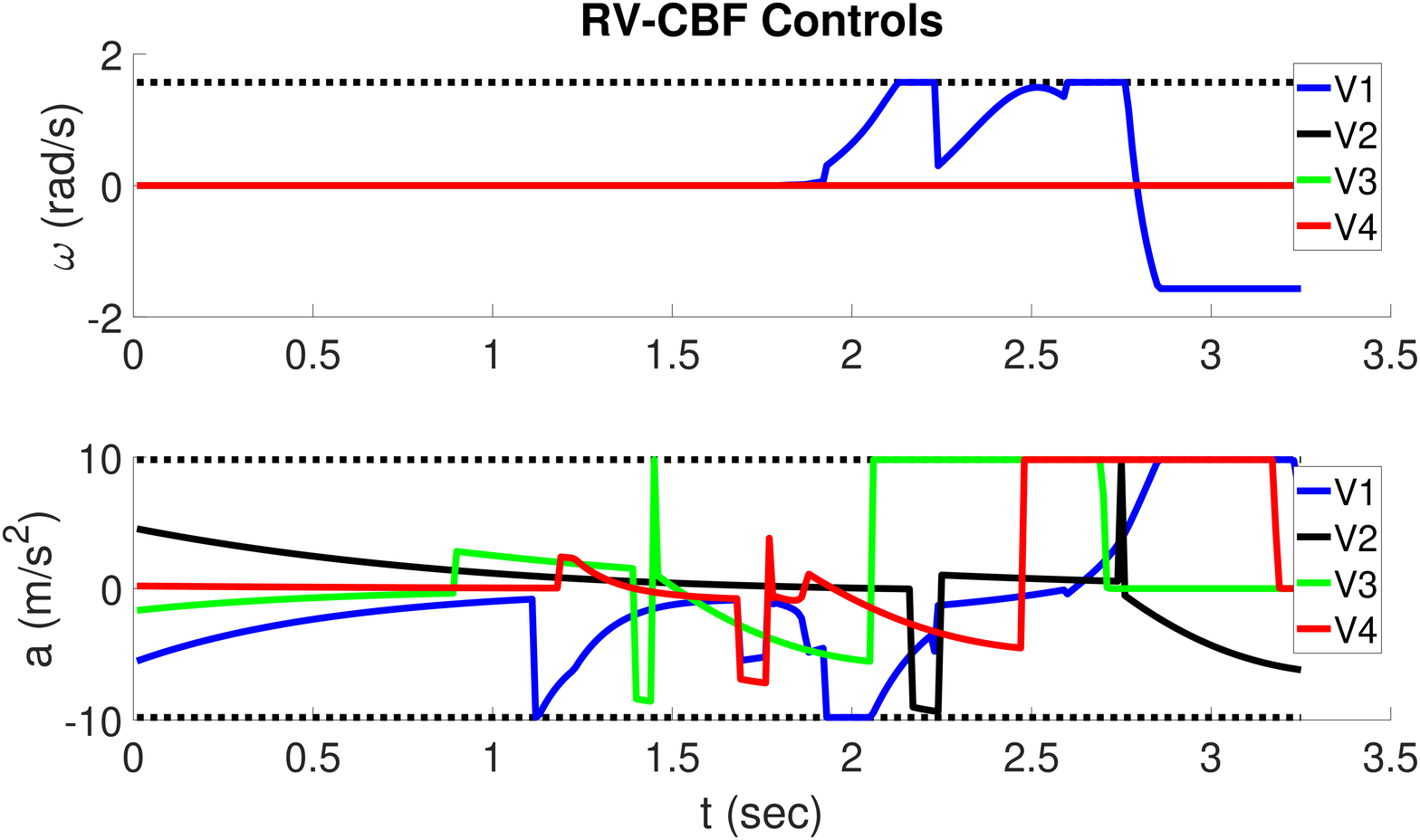}
        \caption{\small{Control solutions for Trial 650 of rff-CBF simulation set.}}
        \label{fig.rvcbf_control_trajectories}
        \vspace{5mm}
    \end{minipage}
    \begin{minipage}[b]{1\columnwidth}
        \centering
        \includegraphics[width=1\columnwidth,clip,trim={0 0 0 12mm}]{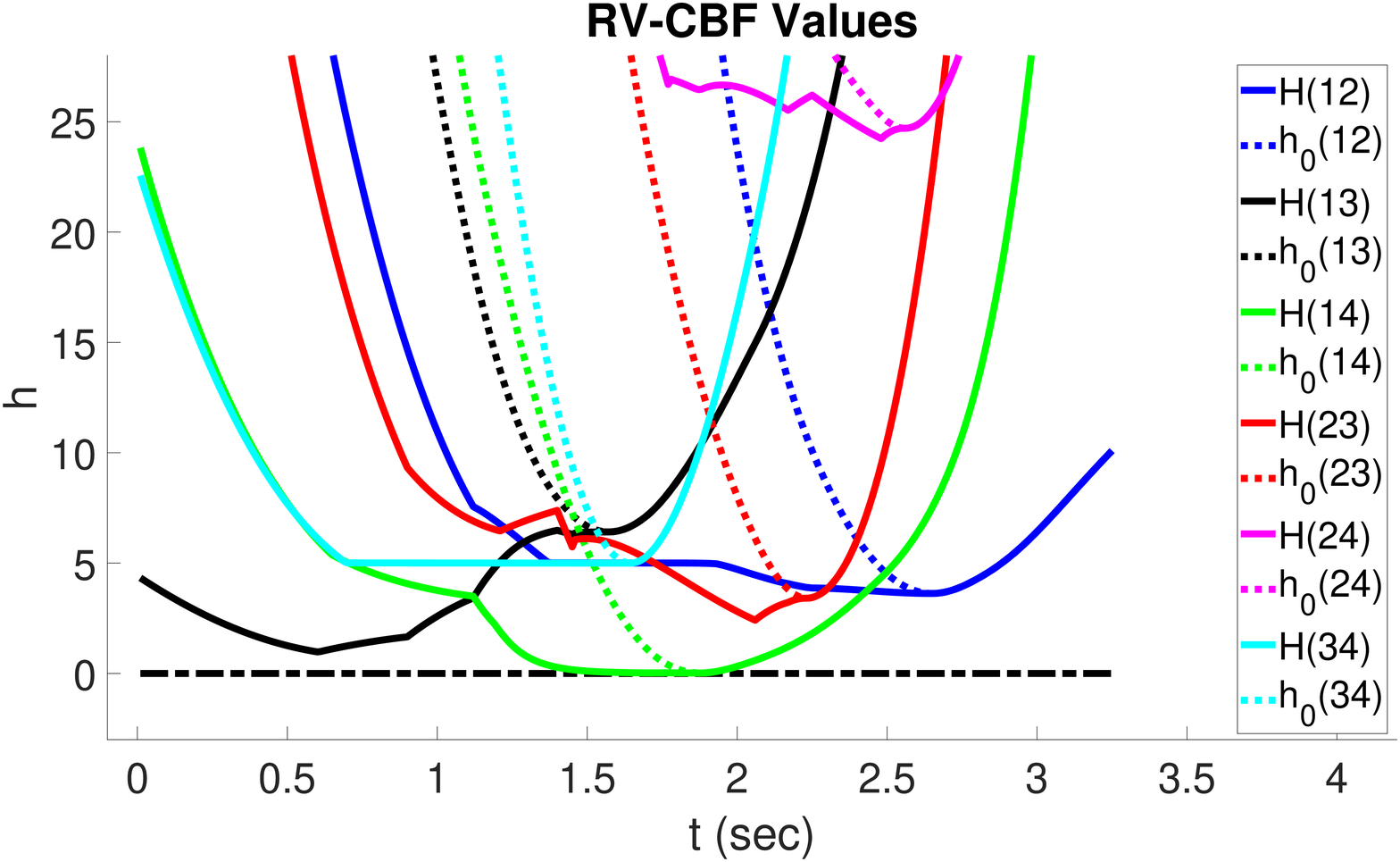}
        \caption{\small{rff-CBF ($H$) and 0-CBF ($h_0$) trajectories for rff-CBF Trial 650. $(ij)$ denote that CBF is evaluated for vehicles $i$ and $j$.}}
        \label{fig.rvcbf_trajectories}
    \end{minipage}
    \vspace{-5mm}
\end{figure}

\subsection{Decentralized Control: Rover Experiments}
We further demonstrated the success of our decentralized rff-CBF-QP controller on a collection of AION R1 UGV rovers in an intersection scenario in the lab. Each of the 5 rovers was asked to proceed straight through the intersection while obeying a speed limit (encoded via \eqref{eq: CBF speed limit}) and avoiding collisions with each other (using rff-CBF \eqref{eq.robust_virtual_cbf}). Assuming the rovers behaved according to the bicycle model \eqref{eq: dynamic bicycle model}, we used a controller of the form \eqref{eq.decentralized_cbf_qp} to compute acceleration $a_i$ and angular rate $\omega_i$ inputs in order to send velocity $v_i(t_{k+1}) = v_i(t_k) + a_i \Delta t$ and $\omega_i$ commands to the rovers' on-board PID controllers. The full control loop ran at a frequency of 20Hz, where the nominal input $\bb{u}_0$ was computed using the LQR law outlined in Appendix \ref{app.LQR}, position feedback was obtained using a Vicon motion capture system, and an extended Kalman filter was used via the on-board PX4 for state estimation.
\begin{figure}[!h]
    \centering
        \includegraphics[width=\columnwidth,clip]{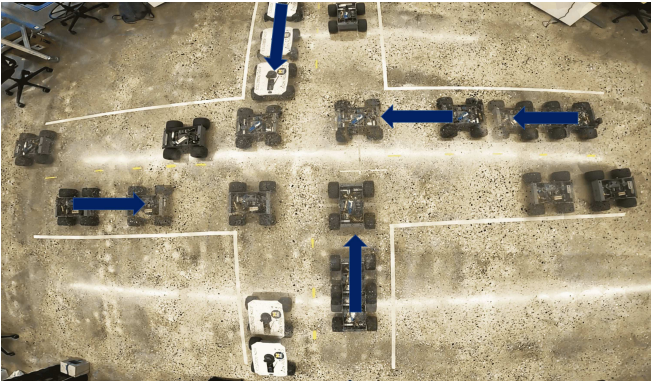}
    \caption{\small{Five rovers safely traverse a four-way intersection in the laboratory environment using a decentralized rff-CBF-QP control law. The rovers at their initial positions are marked with arrows pointing in the direction of motion.}}\label{fig.rovers_skyview}
    \vspace{-3mm}
\end{figure} 

As shown in Figure \ref{fig.rovers_skyview}, our rff-CBF based controller succeeds in driving the vehicles safely through the intersection without encountering a deadlocked situation. The video footage available at the above GitHub link shows that, contrary to behavior expected using traditionally myopic, present-focused CBFs, some rovers accelerated into the intersection in order to avoid predicted future collisions whereas others braked to await their turn.


\section{Conclusion}\label{sec.conclusion}
In this paper, we introduced advancements to traditionally myopic control barrier function based safe control in the form of novel future-focused (ff-) and relaxed future-focused (rff-) CBFs. We then studied their efficacy on a simulated intersection crossing problem for a collection of automobiles modeled as bicycles and controlled by a centralized CBF-QP-based controller under three different CBFs, and discovered that the rff-CBF produced the most favorable empirical results. We further validated our proposed approach on a collection of 5 ground rovers in an intersection scenario in the lab environment. In the future, we plan to further investigate 1) how rff-CBFs and control Lyapunov functions may be combined to make formal guarantees on stabilization and safety, and 2) under what conditions the nominal CBF-QP controller remains feasible when the future-focused CBF-QP may not be.

\bibliographystyle{IEEEtran}
\bibliography{myreferences,refs_intersectionmanagement,refs_carsims,refs_misc,refs_safetycriticalsystems,refs_FTSandFxTS,refs_AdaptiveControl,refs_learningbasedcontrol,refs_marketreports,refs_motionplanners}


\appendices
\section{LQR-based Nominal Control Law}\label{app.LQR}
For each vehicle, we assume that a desired state trajectory, $\bb{q}_i^*(t) = [x_i^* \; y_i^* \; \dot{x}_i^* \; \dot{y}_i^*]^T$, is available. Then, we define the modified state vector and tracking error as $\bb{\zeta}_{i}(t) = [x_i \; y_i \; \dot{x}_i \; \dot{y}_i]^T$, and $\tilde{\bb{\zeta}}_i(t) = \bb{\zeta}_{i}(t)-\bb{q}_i^*(t)$ respectively. We then compute the optimal LQR gain, $K$, for a planar double integrator model and compute $\bb{\mu} = [a_{x,i} \; a_{y,i}]^T = -K\Tilde{\bb{\zeta}}_i$. Then, we map $a_{x,i}$, $a_{y,i}$ to $\omega_i^0$, $a_{i}^0$ via
\begin{align}\nonumber
    \begin{bmatrix}
    \omega_i^0 \\ a_{i}^0
    \end{bmatrix} = S^{-1}\begin{bmatrix}
    a_{x,i} + \dot{y}_i\dot{\psi}_i \\ a_{y,i}-\dot{x}_i\dot{\psi}_i
    \end{bmatrix},
\end{align}
where
\begin{equation}\nonumber
    S = \begin{bmatrix}
    -v_{i}\sin(\psi_i)\sec^2(\beta_i) & \cos(\psi_i) - \sin(\psi_i)\tan(\beta_i) \\ v_{i}\cos(\psi_i)\sec^2(\beta_i) & \sin(\psi_i) + \cos(\psi_i)\tan(\beta_i)
    \end{bmatrix},
\end{equation}
the inverse of which exists as long as $v_{i} \neq 0$. Therefore, if $|v_{i}| < \epsilon$, where $0 < \epsilon \ll 1$, we assign $\omega_i^0 = 0$ and $a_{i}^0 = \sqrt{a_{x,i}^2 + a_{y,i}^2}$.

\end{document}